\documentclass[11pt]{article}
\usepackage{latexsym,amsmath,color,amsthm,amssymb,epsfig,graphicx,mathrsfs}
\usepackage{graphicx}
\usepackage{amssymb}
\usepackage[left=1in,top=1in,right=1in,bottom=1in]{geometry}
\usepackage[linktocpage=true]{hyperref}
\usepackage{setspace}
\usepackage{amssymb, amsmath, amsthm, graphicx,mathrsfs}
\usepackage{caption}

\usepackage{comment}
\def\qed{\hfill\ifhmode\unskip\nobreak\fi\quad\ifmmode\Box\else\hfill$\Box$\fi}
\def\ite#1{\hfill\break${}$\hbox to 50pt {\quad(#1)\hfill}}
\newtheorem{thm}{Theorem}[section]
\newtheorem{cor}[thm]{Corollary}

\newtheorem{lem}[thm]{Lemma}

\newtheorem{claim}[thm]{Claim}

\begin{document}

\title{\vspace{-0.5in}Stability in the Erd\H{o}s--Gallai Theorem on  cycles and paths, II\footnote{This paper started
at SQuaRES meeting of the American Institute of Mathematics.}}

\author{
{{Zolt\'an F\" uredi}}\thanks{
\footnotesize {Alfr\' ed R\' enyi Institute of Mathematics, Hungary.
E-mail:  \texttt{zfuredi@gmail.com.}
{Research was supported in part by grant K116769
from the National Research, Development and Innovation Office NKFIH,
by the Simons Foundation Collaboration Grant \#317487,
and by the European Research Council Advanced Investigators Grant 267195.}
}}
\and
{{Alexandr Kostochka}}\thanks{
\footnotesize {University of Illinois at Urbana--Champaign, Urbana, IL 61801
 and Sobolev Institute of Mathematics, Novosibirsk 630090, Russia. E-mail: \texttt {kostochk@math.uiuc.edu}.
 Research of this author
is supported in part by NSF grant  DMS-1266016
and  by grants 15-01-05867 and 16-01-00499  of the Russian Foundation for Basic Research.
}}
\and
{{Ruth Luo}}\thanks{
\footnotesize {University of Illinois at Urbana--Champaign, Urbana, IL 61801. E-mail: \texttt {ruthluo2@illinois.edu}.
}}
\and{{Jacques Verstra\"ete}}\thanks{Department of Mathematics, University of California at San Diego, 9500
Gilman Drive, La Jolla, California 92093-0112, USA. E-mail: {\tt jverstra@math.ucsd.edu.} Research supported by NSF Grant DMS-1101489. }}

\maketitle

\vspace{-0.3in}


\begin{abstract}

The Erd\H{o}s--Gallai Theorem states that for $k \geq 3$, any $n$-vertex graph with no cycle of length at least $k$ has at most 
$\frac{1}{2}(k-1)(n-1)$ edges. 
A stronger version of the Erd\H{o}s--Gallai Theorem was given by Kopylov: If $G$ is a 2-connected $n$-vertex graph with no cycle of length at least $k$, then
$e(G) \leq \max\{h(n,k,2),h(n,k,\lfloor \frac{k-1}{2}\rfloor)\}$, where
$h(n,k,a) := {k - a \choose 2} + a(n - k + a)$. Furthermore, Kopylov presented the two possible extremal graphs, one with $h(n,k,2)$ edges and one with $h(n,k,\lfloor \frac{k-1}{2}\rfloor)$ edges.

In this paper, we complete a stability theorem which strengthens Kopylov's result. In particular, we show that for $k \geq 3$ odd and all $n \geq k$, every $n$-vertex $2$-connected graph $G$ with no cycle of length at least $k$ is a subgraph of one of the two extremal graphs or $e(G) \leq \max\{h(n,k,3),h(n,k,\frac{k-3}{2})\}$. The upper bound for $e(G)$ here is tight.

\end{abstract}

\section{Introduction}

One of the  basic Tur\'{a}n-type problems 
is to determine the maximum number of edges in an $n$-vertex graph with no $k$-vertex path. 
Erd\H{o}s and Gallai~\cite{ErdGal59} in 1959 proved the following fundamental result on this problem.

\begin{thm}[Erd\H{o}s and Gallai~\cite{ErdGal59}]\label{ErdGallaiPath}
Fix $n,k \geq 2$. If $G$ is an $n$-vertex graph that does not contain a path with $k$ vertices, then $e(G) \leq \frac{1}{2}(k-2)n$.  
\end{thm}

When $n$ is divisible by $k - 1$, the bound is best possible. Indeed, 
 the   $n$-vertex graph whose every component is the complete graph $K_{k-1}$ has $\frac{1}{2}(k-2)n$ edges and
  no $k$-vertex paths.
Also, if $H$ is an $n$-vertex graph without a $k$-vertex path $P_k$, then by adding to $H$ a new vertex $v$ adjacent to all vertices of $H$
  we obtain  an $(n+1)$-vertex graph $H'$
with $e(H)+n$ edges that contains no cycle  of length $k+1$ or longer. Then Theorem~\ref{ErdGallaiPath}  follows from another
theorem of Erd\H{o}s and Gallai:

\begin{thm}[Erd\H{o}s and Gallai~\cite{ErdGal59}]\label{ErdGallaiCyc}
Fix $n,k \geq 3$. If $G$ is an $n$-vertex graph that does not contain a cycle of length at least $k$, then $e(G) \leq \frac{1}{2}(k-1)(n-1)$.  
\end{thm}

The bound of this theorem
 is best possible for $n - 1$ divisible by $k - 2$. Indeed, 
any connected $n$-vertex graph in which every block is a $K_{k-1}$
 has 
 $\frac{1}{2}(k-1)(n-1)$ edges and no cycles of length at least $k$.
In the 1970's, some refinements and new proofs of   Theorems~\ref{ErdGallaiPath} and~\ref{ErdGallaiCyc}
were obtained by   Faudree and Schelp~\cite{FaudScheB,FaudSche75}, Lewin~\cite{Lewin}, Woodall~\cite{Woodall}, and Kopylov~\cite{Kopy} -- see~\cite{FS224} for more details.
The strongest version was proved by
   Kopylov~\cite{Kopy}. His result is stated in terms of
 the following graphs.  Let  $n \geq k$ and $1\leq   a < \frac{1}{2}k$. The $n$-vertex graph $H_{n,k,a}$ is as follows.
 The vertex set of $H_{n,k,a}$ is the union of three disjoint sets $A,B,$ and $C$ such that $|A| = a$, $|B| = n - k + a$ and $|C| = k - 2a$,
 and the edge set of $H_{n,k,a}$ consists of all edges between $A$ and $B$ together with all edges in $A \cup C$ (Fig.~1 shows 
$H_{14,11,3}$). 
Let
\[ h(n,k,a) := e(H_{n,k,a}) = {k - a \choose 2} + a(n - k + a).\]
\begin{figure}\label{figu1}
  \centering
    \includegraphics[width=0.25\textwidth]{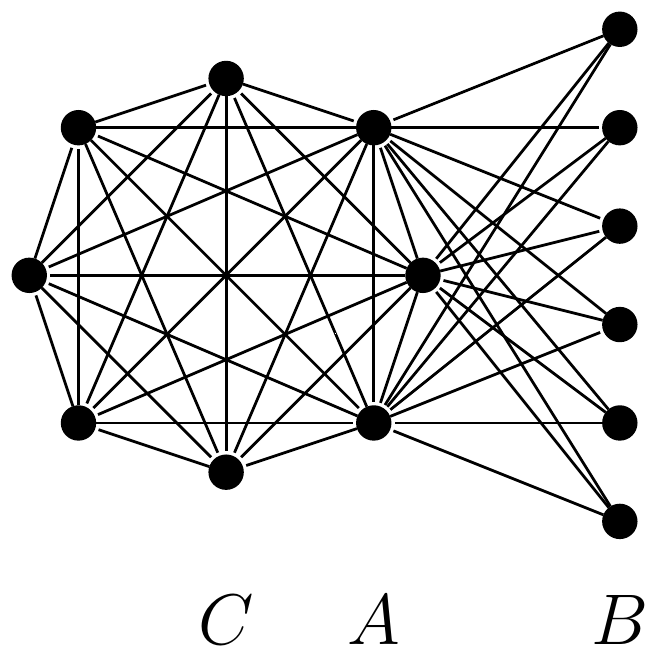}
  \caption{$H_{14,11,3}$.}
\end{figure}

For a graph $G$, let $c(G)$ denote the length of a longest cycle in $G$. Observe that
 $c(H_{n,k,a}) < k$: Since $|A\cup C|=k-a$, any cycle $D$
 of at length at least $k$ has
at least $a$ vertices in $B$. But as $B$ is independent and $2a<k$, $D$ also has to contain 
at least $k+1$ neighbors of the vertices in $B$, while only $a$ vertices in  $A$ have neighbors in $A$.
Kopylov~\cite{Kopy} showed that the extremal $2$-connected  $n$-vertex graphs with no cycles
of length at least $k$ are
$G = H_{n,k,2}$ and $G = H_{n,k,t}$: the first has more edges for small $n$, and the second --- for large $n$.

\begin{thm} [Kopylov  \cite{Kopy}] \label{th:Kopylov2}
Let $n \geq k \geq 5$ and $t = \lfloor \frac{1}{2}(k-1)\rfloor$. If $G$ is an $n$-vertex $2$-connected graph with $c(G)<k$, then
\begin{equation}\label{eq:kop}
   e(G)\leq \max\{h(n,k,2), h(n,k,t)\}
\end{equation}
with equality only if $G = H_{n,k,2}$ or $G = H_{n,k,t}$.
 \end{thm}
 
\section{Main results}

\subsection{A previous result}
 Recently, three of the present authors proved in~\cite{main}  a stability version of Theorems~\ref{ErdGallaiCyc} and~\ref{th:Kopylov2} for $n$-vertex $2$-connected graphs with  $n\geq 3k/2$, but
the problem remained open for $n<3k/2$ when $k\geq 9$.
 The main result of~\cite{main} was the following:

\begin{thm}[F\"uredi, Kostochka, Verstra\"ete~\cite{main}]\label{oldmainthm}Let $t \geq 2$ and $n \geq 3t$ and $k \in \{2t + 1,2t + 2\}$. Let $G$ be a
 $2$-connected $n$-vertex graph $c(G) < k$. Then
$e(G) \leq h(n,k,t-1)$ unless
\vspace{-4mm}
\begin{center}
\begin{tabular}{lp{5.8in}}
$(a)$ & $k = 2t + 1$, $k \neq 7$,  and $G \subseteq  H_{n,k,t}$ or \\
$(b)$ & $k = 2t + 2$ or $k= 7$,  and $G - A$ is a star forest for some $A \subseteq V(G)$ of size at most $t$. \\
\end{tabular}
\end{center}
\end{thm}

Note that \begin{equation*}\label{j13}
  h(n,k,t) - h(n,k,t-1) = \left\{\begin{array}{ll}
  
n - t - 3 & \mbox{ if }k = 2t + 1,\\
n - t - 5 & \mbox{ if }k = 2t + 2.
\end{array}\right.
\end{equation*}

The paper~\cite{main} also describes the $2$-connected $n$-vertex graphs with $c(G) < k \leq 8$ for all $n\geq k$. 

\subsection{The essence of the main result}
Together with~\cite{main},
 this paper  gives a full description of the 2-connected $n$-vertex graphs with $c(G) < k$ and `many' edges for all $k$ and $n$.
 Our main result is:

\begin{thm}\label{mainthm}
Let $t \geq 4$ and $k \in \{2t+1, 2t+2\}$, so that $k \geq 9$. 
If $G$ is a $2$-connected graph on $n \geq k$ vertices and $c(G) < k$, then either $e(G) \leq \max\{h(n,k,t-1), h(n,k,3)\}$ or  
\vspace{-4mm}
\begin{center}
\begin{tabular}{lp{5.8in}}
$(a)$ & $k = 2t + 1$ and $G \subseteq  H_{n,k,t}$ or \\
$(b)$ & $k = 2t + 2$  and $G - A$ is a star forest for some $A \subseteq V(G)$ of size at most $t$. \\
$(c)$ & $G \subseteq H_{n,k,2}$.
\end{tabular}
\end{center}
\end{thm} 

 \begin{figure}[!ht]
  \centering
    \includegraphics[height=0.15\textwidth]{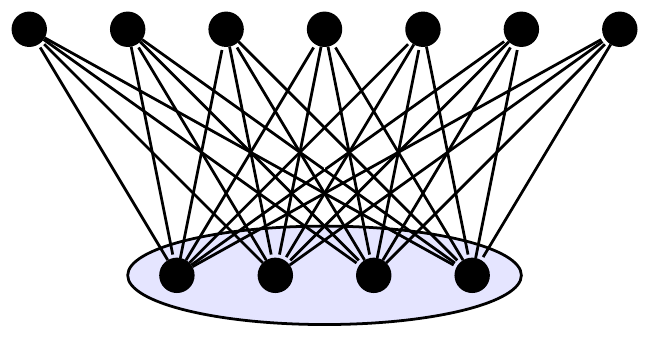}\;\;\;
    \includegraphics[height=0.15\textwidth]{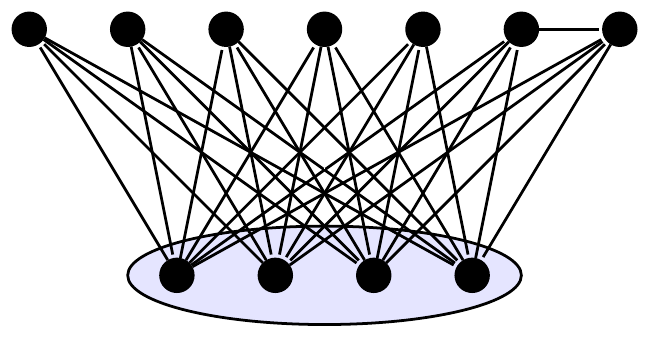}\;\;\;
    \includegraphics[height=0.15\textwidth]{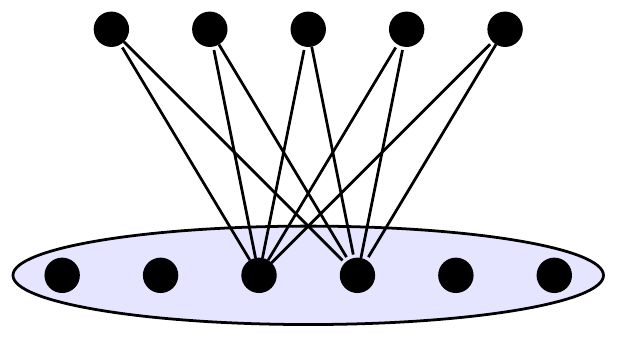}
  \caption{$H_{n,k,t} (k = 2t+1), H_{n,k,t} (k = 2t + 2), H_{n,k,2}$;\\
${}$ \quad\quad\quad\quad   ovals denote complete subgraphs of order $t$, $t$, and $k-2$ respectively.}
\end{figure}

Note that the case $n<k$ is trivial and the case $k\leq 8$ was fully resolved in~\cite{main}.

\subsection{A more detailed form of the main result}
In order to prove Theorem~\ref{mainthm}, we need a more detailed description of graphs satisfying (b) in the theorem that do not contain
`long' cycles.


Let ${\mathcal G}_1(n,k) = \{H_{n,k,t},H_{n,k,2}\}$. 
Each $G\in {\mathcal G}_2(n,k)$ is defined by a partition $V(G)=A\cup B\cup C$ and two vertices $a_1\in A$, $b_1\in B$
 such that  $|A| = t$, $G[A]=K_t$, $G[B]$ is the empty graph, $G(A,B)$ is a complete bipartite graph, and 
 $N(c)=\{a_1,b_1\}$ for every $c \in C$.
 Every member of $G\in {\mathcal G}_3(n,k)$ is defined by a partition $V(G)=A\cup B\cup J$
 such that  $|A|=t$, $G[A]=K_t$, $G(A,B)$ is a complete bipartite graph, and
\newline ${}\quad$ 
---  $G[J]$ has more than one component, \\
 ${}\quad$ --- all components of $G[J]$ are stars with at least two vertices each, \\
 ${}\quad$ ---  there is a $2$-element subset $A'$ of $A$ such that $N(J)\cap (A\cup B)=A'$, \\
 ${}\quad$ --- for every component $S$ of $G[J]$ with at least $3$ vertices, all leaves of $S$ have degree 2 in $G$ and are adjacent to the same vertex $a(S)$ in $A'$.

The class $\mathcal{G}_4(n,k)$ is empty unless $k = 10$. Each graph $H \in {\mathcal G}_4(n,10)$ has a $3$-vertex set $A$ such that $H[A]=K_3$ and $H - A$ is a star forest such that
if a component $S$ of $H - A$ has more than two vertices then all its leaves have degree 2 in $H$ and are adjacent to the same vertex $a(S)$ in $A$.
These classes are illustrated below:
\begin{figure}\label{nfi}
\begin{center}
\includegraphics[width=6.5in]{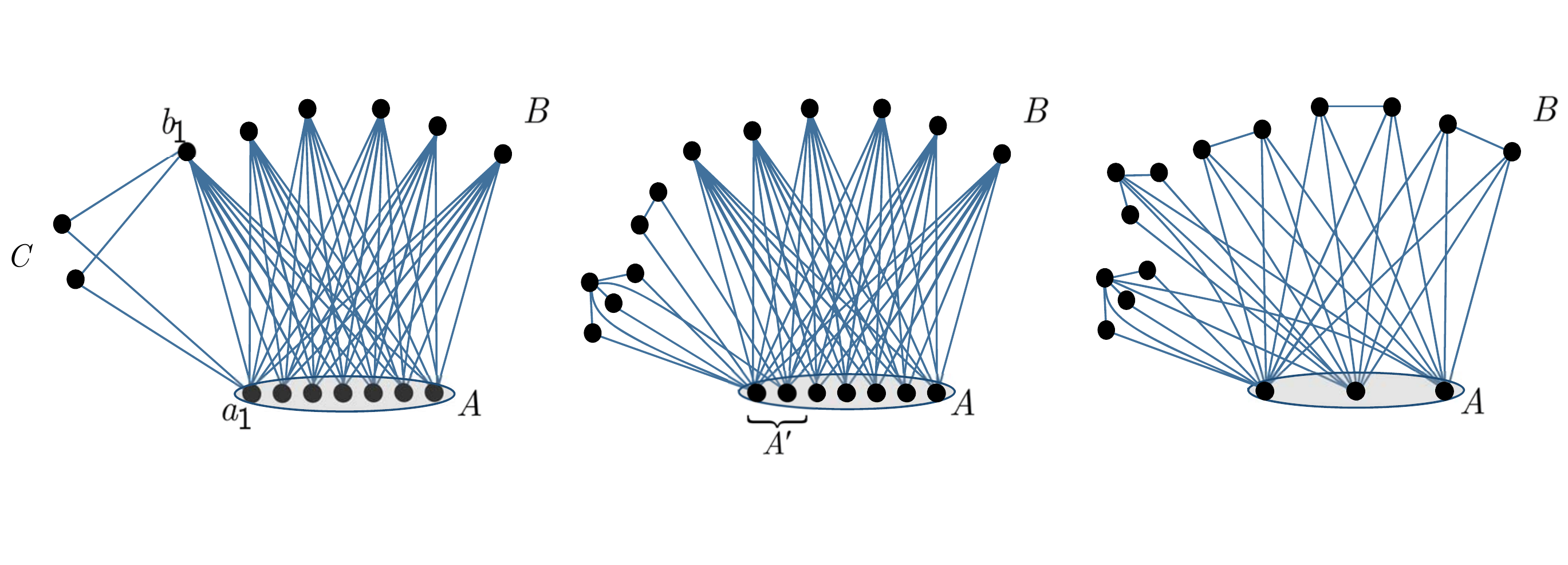}
\caption{Classes $\mathcal{G}_2(n,k)$, $\mathcal{G}_3(n,k)$ and $\mathcal{G}_4(n,10)$.}
\end{center}
\end{figure}




We can refine Theorem~\ref{mainthm} in terms of
the classes $\mathcal{G}_i(n,k)$ as follows:

\begin{thm}\label{main} {\rm (Main Theorem)}
Let $k \geq 9$, $n \geq k$ and $t=\left\lfloor \frac{1}{2}(k-1)\right\rfloor$. Let $G$ be an  $n$-vertex $2$-connected graph  with no cycle of length at least $k$. Then
 $e(G)\leq \max\{h(n,k,t-1),h(n,k,3)\}$  or  $G$ is a subgraph of a graph in $\mathcal{G}(n,k)$, where
\vspace{-4mm} 
\begin{center}
 \begin{tabular}{lp{5.8in}}
$(1)$ & if  $k$ is odd, then  $\mathcal{G}(n,k)= \mathcal{G}_1(n,k)=\{H_{n,k,t},H_{n,k,2}\}$;\\
$(2)$ & if $k$ is even and $k \neq 10$, then  $\mathcal{G}(n,k)=\mathcal{G}_1(n,k) \cup \mathcal{G}_2(n,k) \cup \mathcal{G}_3(n,k)$; \\
$(3)$ & if  $k = 10$, then  $\mathcal{G}(n,k)=\mathcal{G}_1(n,10) \cup \mathcal{G}_2(n,10) \cup \mathcal{G}_3(n,10) \cup \mathcal{G}_4(n,10)$.
\end{tabular}
\end{center}
\end{thm}

Since every graph in $\mathcal{G}_2(n,k)\cup \mathcal{G}_3(n,k)$ and many graphs in $\mathcal{G}_4(n,k)$ have a separating set of size $2$ (see
Fig.~3),
 the theorem implies the following simpler statement for $3$-connected graphs:
\begin{cor}\label{maincor}
Let $k \in \{2t+1, 2t+2\}$ where $k \geq 9$. If $G$ is a $3$-connected graph on $n \geq k$ vertices and $c(G) < k$,
 then either $e(G) \leq \max\{h(n,k,t-1), h(n,k,3)\}$ or $G \subseteq H_{n,k,t}$ or $k=10$ and $G$ is a subgraph of some  graph $H \in \mathcal G_4(n,10)$ such that each component of $H - A$ has at most 2 vertices. 
\end{cor} 

\section{The proof idea}
\subsection{Small dense subgraphs}
 First we define some more graph classes. For a graph $F$ and a nonnegative integer $s$, we denote by $\mathcal{K}^{-s}(F)$ the family of graphs obtained from $F$ by deleting at most $s$ edges.

Let  $F_0=F_0(t)$ denote the complete bipartite graph $K_{t,t+1}$ with partite sets $A$ and $B$ where $|A|=t$ and $|B|=t+1$.
Let $\mathcal{F}_0=\mathcal{K}^{-t+3}(F_0)$, i.e., the family of subgraphs of $K_{t,t+1}$ with at least $t(t+1)-t+3$ edges.

Let $F_1=F_1(t)$ denote the complete bipartite graph $K_{t,t+2}$ with partite sets $A$ and $B$ where $|A|=t$ and $|B|=t+2$.
Let $\mathcal{F}_1=\mathcal{K}^{-t+4}(F_1)$, i.e., the family of subgraphs of $K_{t,t+2}$ with at least $t(t+2)-t+4$ edges.

Let $\mathcal{F}_2$ denote the family of graphs obtained from a graph in  $\mathcal{K}^{-t+4}(F_1)$
by subdividing an edge $a_1b_1$ with a new vertex $c_1$, where $a_1\in A$ and $b_1\in B$.
Note that any member $H\in \mathcal{F}_2$ has at least $|A||B|-(t-3)$ edges between $A$ and $B$ and the
 pair $a_1b_1$ is not an edge.

Let $F_3=F_3(t,t')$ denote the complete bipartite graph $K_{t,t'}$ with partite sets $A$ and $B$ where $|A|=t$ and $|B|=t'$.
Take a graph from  $\mathcal{K}^{-t+4}(F_3)$,  select two non-empty subsets $A_1$, $A_2\subseteq  A$ with $|A_1\cup A_2|\geq 3$
 such that $A_1\cap A_2=\emptyset$ if $\min\{ |A_1|, |A_2|\}=1$,
 add two vertices $c_1$ and $c_2$, join them to each other and add the edges from $c_i$ to the elements of $A_i$, ($i=1,2$).
The class of obtained graphs is denoted by ${\mathcal F}(A,B,A_1, A_2)$.
The family $\mathcal{F}_3$ consists of these graphs when $|A|=|B|=t$,
 $|A_1|=|A_2|=2$ and $A_1\cap A_2=\emptyset$. In particular, $\mathcal{F}_3(4)$ consists of exactly one graph, call it $F_3(4)$.

Graph $F_4$ has vertex set $A\cup B$, where $A=\{a_1,a_2,a_3\}$ and $B:=\{ b_1, b_2, \dots , b_6\}$ are disjoint. Its  edges
are the edges of the complete bipartite graph $K(A,B)$ and three extra edges $b_1b_2$, $b_3b_4$, and $b_5b_6$ (see Fig.~4 below).
Define $F_4'$ as the (only) member of ${\mathcal F}(A,B,A_1, A_2)$ such that $|A|=|B|=t=4$, $A_1=A_2$, and $|A_i|=3$.
Let $\mathcal{F}_4:= \{ F_4, F'_4\}$, which is defined only for $t=4$.

 \begin{figure}[!ht]\label{fr}
  \centering
    \includegraphics[height=0.15\textwidth]{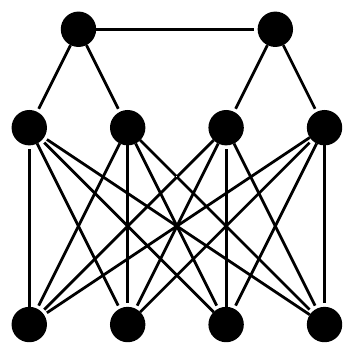}\;\;\;
    \includegraphics[height=0.15\textwidth]{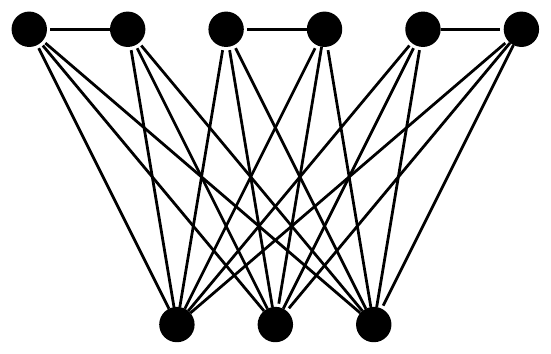}\;\;\;
    \includegraphics[height=0.15\textwidth]{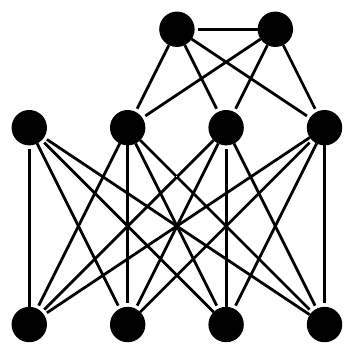}
  \caption{Graphs $F_3(4), F_4,$ and $F_4'$.}
\end{figure}

Define $\mathcal F(k):=\left\{\begin{array}{ll} \mathcal F_0, &\mbox{if $k$ is odd},\\
\mathcal{F}_1\cup\dots \cup \mathcal{F}_4, &  \mbox{if $k$ is even.}\end{array}\right.$

%
%



\subsection{Proof idea}
For our proof, it will be easier to use the stronger induction assumption that the graphs in question contain certain dense graphs from $\mathcal F(k)$. We will prove the following slightly stronger version of Theorem~\ref{main} which also implies Theorem~\ref{mainthm}. 

{\bf Theorem~\ref{main}$'$}\label{mainthm2}
\emph{Let $t\geq 4$, $k \in \{2t+1, 2t+2\}$, and $n \geq k$. Let $G$ be an  $n$-vertex $2$-connected graph  with no cycle of length at least $k$. Then
 $e(G)\leq \max\{h(n,k,t-1), h(n,k,3)\}$  or  
\vspace{-4mm}
 \begin{center}
\begin{tabular}{lp{5.8in}}
$(a)$ & $G \subseteq H_{n,k,2}$, or \\
$(b)$ & $G$ is contained in a graph in $\mathcal G(n,k) - \{H_{n,k,2}\}$, and $G$ contains a subgraph $H \in \mathcal F(k)$. \\
\end{tabular}
\end{center}}
%

The method of the proof is a variation of that of~\cite{main}. Also, when $n$ is close to $k$, we use Kopylov's disintegration method. We take an $n$-vertex graph $G$ satisfying the hypothesis of Theorem~\ref{main}$'$, and iteratively contract edges in a certain way so that each intermediate graph still satisfies the hypothesis. We consider the final graph of this process $G_m$ on $m$ vertices and show that $G_m$ satisfies Theorem~\ref{main}$'$. Two results from~\cite{main} will be instrumental. The first is:


\begin{lem}[Main lemma on contraction~\cite{main}]\label{oldmain2}  Let $k\geq 9$ and suppose $F$ and $F'$ are $2$-connected 
graphs such that $F = F'/xy$ and $c(F') < k$. If $F$ contains a subgraph $H \in \mathcal{F}(k) $, then $F'$ also contains 
a subgraph $H' \in \mathcal{F}(k) $.\qed \end{lem}

This lemma shows that if $G_m$ contains a subgraph $H \in \mathcal F(k)$, then the original graph $G$ also contains a subgraph in $ \mathcal{F}(k) $. 
The second result (proved in Subsection 4.5 of~\cite{main}) is:
\begin{lem}[\cite{main}]\label{oldmain3} Let $k \geq 9$, and let $G$ be a $2$-connected graph with $c(G) < k$ 
and $e(G) > h(n,k,t-1)$. If $G$ contains a subgraph $H \in \mathcal{F}(k) $, then $G$ is a subgraph of a graph in $\mathcal G(n,k) - \{H_{n,k,2}\}$.\qed \end{lem}

We will split the proof into the cases of small $n$ and large  $n$. The following observations can be obtained by simple calculations (for $t\geq 4$):

\vspace{-1mm}
\begin{center}
    \begin{tabular}{| c | c | c |}
    \hline
    $k$ & $h(n,k,3) \geq h(n,k,t-1)$ &  $h(n,k,2) \geq  h(n,k,t-1)$ \\  \hline
    $2t+1$ & If and only if $n \leq k+(t-5)/2$ & If and only if $n \leq k + t/2 -1$ \\
    $2t+2$ & If and only if $n \leq k+(t-3)/2$ & If and only if $n \leq k + t/2$ \\
    \hline
    \end{tabular}
\end{center}

In the case of large $n$  we will contract an edge such that the new graph still has more than $h(n-1,k,t-1)$ edges. 
In order to apply induction, we also need the number of edges to be greater than $h(n-1, k, 3)$. 
To guarantee this, we pick the cutoffs for the two cases $n \leq k + (t-1)/2$ and $n > k + (t-1)/2$ (therefore $n-1 > k + (t-3)/2$).


\section{Tools}\label{tools}
\subsection{Classical theorems}

\begin{thm}[Erd\H{o}s~\cite{Erd62}]\label{th:er}
Let $d\geq 1$ and $n>2d$ be integers, and
\[ \ell_{n,d}=\max\left\{\binom{n-d}{2}+d^2,\binom{\lceil\frac{n+1}{2}\rceil}{2}+{\Big\lfloor\frac{n-1}{2}\Big\rfloor}^2\right\}.\]
Then every $n$-vertex graph $G$ with $\delta(G)\geq d$ and $e(G) > \ell_{n,d}$ is hamiltonian.\qed
  \end{thm}

\begin{thm}[Chv\' atal~\cite{Ch}]\label{t1} Let $n\geq 3$ and $G$ be an $n$-vertex graph with vertex degrees $d_1\leq d_2\leq\ldots\leq d_n$.
If $G$ is not hamiltonian, then there is some $i<n/2$ such that $d_i\leq i$ and $d_{n-i}<n-i$.\qed
\end{thm}

\begin{thm}[Kopylov~\cite{Kopy}]\label{le:kop}
If $G$ is $2$-connected and $P$ is an $x,y$-path of $\ell$ vertices, then $c(G)\geq \min \{\ell,d(x,P)+ d(y,P)\}$.\qed
\end{thm}


\subsection{Claims on contractions}

A helpful tool will be 
the following lemma from~\cite{main} on contraction.

\begin{lem}[\cite{main}]\label{con} Let $n\geq 4$ and let $G$ be an  $n$-vertex $2$-connected graph. 
For every $v \in V(G)$, there exists $w \in N(v)$ such that $G/vw$ is $2$-connected. 
 \qed 
\end{lem}


For an edge $xy$ in a graph $H$, let $T_H(xy)$ denote the number of triangles containing $xy$.  Let
$T(H)=\min\{T_H(xy)\,:\,xy\in E(H)\}$. When we contract an edge $uv$ in a graph $H$, the degree of every
$x\in V(H)-u-v$ either does not change or decreases by $1$. Also the degree of $u*v$ in $H/uv$ is at least
$\max\{d_H(u),d_H(v)\}-1$. Thus
\begin{equation}\label{m25}
\mbox{$d_{H/uv}(w)\geq d_H(w)-1$ for any  $w\in V(H)$ and $uv\in E(H)$.
Also $d_{H/uv}(u*v)\geq d_H(u)-1$.}
\end{equation}
Similarly,
\begin{equation}\label{m251}
\mbox{$T(H/uv)\geq T(H)-1$ for every graph $H$ and $uv\in E(H)$.}
\end{equation}

We will use the following analog of Lemma 3.3 in~\cite{main}.

\begin{lem}\label{hhh}  Let $h$ be a positive integer. Suppose  a $2$-connected graph $G$  is obtained from a  $2$-connected graph $G'$ by contracting edge $xy$ into $x*y$
chosen using the following rules:\\
$(i)$ one of $x,y$, say $x$ is a vertex of the minimum degree in $G'$;\\
$(ii)$ $T_{G'}(xy)$ is the minimum among the edges $xu$ incident with $x$ such that $G'/xu$ is $2$-connected. 
 (Such edges exist by
Lemma~\ref{con}).
If $G$ has at least $h$ vertices of degree at most $h$, then either $G'=K_{h+2}$ or \\
$(a)$ $G'$ also has a vertex of degree at most $h$, and \\
$(b)$ $G'$ has at least $h+1$ vertices of degree at most $h+1$.
\end{lem}

{\bf Proof.} Since $G$ is $2$-connected, $h\geq 2$. Let $V_{\leq s}(H)$ denote the set of vertices of degree at most $s$ in $H$.
Then by~\eqref{m25}, each $v\in V_{\leq h}(G)-x*y$  is also in $V_{\leq h+1}(G')$. Moreover, then by (i), 
\begin{equation}\label{m28}
x\in V_{\leq h+1}(G').
\end{equation}
Thus if $x*y\notin V_{\leq h}(G)$, then (b) follows. But if $x*y\in V_{\leq h}(G)$, then by~\eqref{m25}, also $y\in V_{\leq h+1}(G')$.
So, again (b) holds.

If $V_{\leq h-1}(G)\neq\emptyset$, then (a) holds by~\eqref{m25}.
So, if (a) does not hold, then 
\begin{equation}\label{eq1'}
\mbox{
each $v\in V_{\leq h}(G)-x*y$ has degree $h+1$ in $G'$ and is adjacent to both $x$ and $y$ in $G'$.}
\end{equation}

{\bf Case 1:} $|V_{\leq h}(G)-x*y|\geq h$. Then by~\eqref{m28},  $d_{G'}(x)=h+1$.
This in turn yields $N_{G'}(x)=V_{\leq h}(G)+y$.
Since $G'$ is $2$-connected, each
$v\in V_{\leq h}(G)-x*y$ is not a cut vertex.
Furthermore, $\{x,v\}$ is not a cut set. If it was,
because $y$ is a common neighbor of all neighbors of $x$,  all neighbors of $x$ must be in the same
component as $y$ in $G'-x-v$. It follows that
\begin{equation}\label{eq:8}
\mbox{
for every $v\in V_{\leq h}(G)-x*y$,  $G'/vx$ is $2$-connected.}
\end{equation}

If $uv\notin E(G)$ for some $u,v\in V_{\leq h}(G)$, then by~\eqref{eq:8} and (i), we would contract the edge $xu$
rather than $xy$. Thus $G'[V_{\leq h}(G)\cup \{x,y\}]=K_{h+2}$ and so
either $G'=K_{h+2}$ or $y$ is a cut vertex in $G'$,
as claimed.

{\bf Case 2:} $|V_{\leq h}(G)-x*y|=h-1$. Then $x*y\in V_{\leq h}(G)$. This means $d_{G'}(x)=d_{G'}(y)=h+1$
and $N_{G'}[x]=N_{G'}[y]$. So by~\eqref{eq1'}, there is $z\in V(G)$ such that $N_{G'}[x]=N_{G'}[y]=V_{\leq h}(G)\cup \{x,y,z\}$.
Again~\eqref{eq:8} holds (for the same reason that $N_{G'}[x]\subseteq N_{G'}[y]$).
Thus similarly $vu\in E(G')$ for every $v\in V_{\leq h}(G)-x*y$ and every $u\in V_{\leq h}(G)+z$. Hence
$G'[V_{\leq h}(G)\cup \{x,y,z\}]=K_{h+2}$ and  either $G'=K_{h+2}$ or $z$ is a cut vertex in $G'$,
as claimed.\qed

\subsection{A property of graphs in $\mathcal{F}(k) $}


A useful feature of graphs in  $\mathcal{F}(k) $ is the following.

\begin{lem}\label{co28}
Let $k \geq 9$ and $n\geq k$. Let $F$ be an $n$-vertex graph contained in $H_{n,k,t}$ with $e(F)>h(n,k,t-1)$.
 Then $F$ contains a graph in $\mathcal{F}(k)$.
 \end{lem}

\begin{proof}
Assume the sets $A,B,C$ to be as in the definition of $H_{n,k,t}$. We will use induction on $n$.

{\em Case 1:} $k=2t+1$.  If $n=k$, then  $F\in\mathcal{K}^{-t+3}(H_{k,k,t})$ because
$h(k,k,t)-h(k,k,t-1)-1=t-3$.
 Thus, since $H_{k,k,t}\supseteq F_0(t)$, $F$ contains a subgraph in $\mathcal{F}_0$.
Suppose now the lemma holds for all $k\leq n'<n$. If $\delta(F)\geq t$, then
each $v\in V(F)-A$ is adjacent to every $u\in A$.
 Hence $F$ contains $K_{t,n-t}$. If $\delta(F)<t$, then since $A$ is dominating and $n>2t$, there is $v\in V(F)-A$
 with $d_F(v)\leq t-1$. Then $F-v\subseteq H_{n-1,k,t}$, and we are done by induction.

{\em Case 2:}  $k=2t+2$. Let $C=\{c_1,c_2\}$. If $n=k$ then as in Case 1,
 $$e(H_{k,k,t})-e(F) \leq h(k,k,t)-h(k,k,t-1)-1=t - 4,$$
  i.e.,  $F\in\mathcal{K}^{-t+4}(H_{k,k,t})$.
Since $F_1(t)\subseteq  H_{k,k,t}$, $F$ contains a subgraph in $\mathcal{F}_1$.
Suppose now the lemma holds for all $k\leq n'<n$.
If $\delta(F)<t$, then there is $v\in V(F) - A$
 with $d_F(v)\leq t-1$. Then $F-v\subseteq H_{n-1,k,t}$, and we are done by induction.

Finally, suppose $\delta(F) \geq t$. So, each $v\in B$ is adjacent to every $u\in A$ and each of $c_1,c_2$ has
at least $t-1$ neighbors in $A$. Since $|B \cup \{c_1\}|\geq n-t-1\geq t+2$, $F$ contains a member of $\mathcal{K}^{-1}(F_1(t))$.
Thus $F$ contains a member of $\mathcal{F}_1$ unless $t=4$, $n=2t+3$ and $c_1$ has a nonneighbor $x\in A$. But then $c_1c_2\in E(F)$,
and so $F$ contains either $F_3(4)$ or $F'_4$.
\end{proof}

\section{Proof of Theorem~\ref{main}$'$}\label{mproof}

\subsection{Contraction procedure} 

If $n > k$, we iteratively construct  a sequence of graphs $G_n, G_{n-1}, ... G_{m}$ where $|V(G_j)| = j$. 
In~\cite{main}, the following {\bf Basic Procedure} (BP) was used:

At the beginning of each round, for some $j : k\leq j\leq n$, we have a $j$-vertex $2$-connected  graph $G_j$ with $e(G_j)> h(j,k,t-1)$. 
\begin{center}
\begin{tabular}{lp{5in}}
(BP1) & If $j=k$, then we stop. \\
(BP2) &  If there is an edge $uv$ with $T_{G_j}(uv)\leq t-2$ such that $G_j/uv$ is $2$-connected, choose one such edge so that\\
& $(i)$ $T_{G_j}(uv)$ is minimum, and subject to this \\
& $(ii)$ $uv$ is incident to a vertex of minimum possible degree.\\
& Then obtain $G_{j-1}$ by contracting $uv$. \\
(BP3) & If (BP2) does not hold, $j \geq k + t - 1$ and there is $xy \in E(G_j)$ such that $G_j - x - y$ has at least $3$ components and one of the components, say $H_1$ is a $K_{t-1}$,
then let $G_{j-t+1}=G_j-V(H_1)$. \\
(BP4) & If neither (BP2) nor (BP3) occurs, then we stop.
\end{tabular}
\end{center}

{\bf Remark 5.1.} By definition, (BP3) applies only when $j \geq k + t - 1$. As observed in~\cite{main}, if $j\leq 3t-2$, then a $j$-vertex graph $G_j$ 
with a $2$-vertex set $\{x,y\}$ separating the graph into at least $3$ components cannot have $T_{G_j}(uv)\geq t-1$ for every edge $uv$. It also
was calculated there that if $3t-1\leq j\leq 3t$, then any $j$-vertex graph $G'$ with such $2$-vertex set $\{x,y\}$ and $T_{G'}(uv)\geq t-1$ for every edge $uv$
has at most $h(j,k,t-1)$ edges and so cannot be $G_j$.

In this paper, we also use a quite similar {\bf Modified Basic Procedure} (MBP): start with a $2$-connected, $n$-vertex graph $G=G_n$ with $e(G) > h(n,k,t-1)$ and $c(G) < k$. Then
\begin{center}
\begin{tabular}{lp{5in}}
(MBP0) & if $\delta(G_j) \geq t$, then  apply the rules (BP1)--(BP4) of  (BP) given above;\\
(MBP1) & if $\delta(G_j) \leq t-1$ and $j = k$, then stop;\\
(MBP2) & otherwise, pick a vertex $v$ of smallest degree,
  contract an edge $vu$  with the \\ & minimum $T_{G_j}(vu)$ among the edges $vu$  such that $G_j/vu$ is $2$-connected, and 
\\& set $G_{j-1} = G_j/uv$.
\end{tabular}
\end{center}

\subsection{ Proof of Theorem~\ref{main}$'$ for the case  $n \leq k + (t-1)/2$}
Apply to $G$ the Modified Basic Procedure (MBP) starting from $G_n=G$.
By Remark~1, (BP3) was never applied, since $ k + (t-1)/2<k+t-1$.
 Therefore at every step, we only contracted an edge. 
Denote by $G_m$ the terminating graph of (MBP). Then $G_j$ is $2$-connected and $c(G_j) \leq c(G) < k$
for each $m\leq j\leq n$.
By construction, after each contraction, 
we lose at most $t-1$ edges. It follows that $e(G_m) > h(m, k, t-1)$.


If $m>k$, then the same argument as in~\cite{main} gives us the following structural result:

\begin{lem}[\cite{main}]\label{oldmain4}Let $m > k \geq 9$ and $n \geq k$.
\vspace{-3mm}
\begin{itemize}
\item If $k \neq 10$, then $G_m \subseteq H_{m,k,t}$.
\item If $k=10$, then $G_m \subseteq H_{m,k,t}$ or $G_m \supseteq F_4$. \qed
\end{itemize}
\end{lem}
If $k=10$ and $G_m \supseteq F_4$, then $G_m$ contains a subgraph in $\mathcal F(k)$. 
Otherwise, by Lemma~\ref{co28},  again $G_m$ has a subgraph in $\mathcal F(k)$. 
Next, undo the contractions that were used in (MBP). At every step for $m \leq j \leq n$, by Lemma~\ref{oldmain2}, $G_j$ contains some subgraph 
$H' \in \mathcal F(k)$. In particular, $G = G_n$ contains such a subgraph. Thus by Lemma~\ref{oldmain3}, we get our result.
So, below we assume
\begin{equation}\label{m=k}
 m = k.
\end{equation}

Since $c(G_k) < k$, $G_k$ does not have a hamiltonian cycle. Denote the vertex degrees of $G_k$ $d_1 \leq d_2 \leq ... \leq d_k$. 
By Theorem~\ref{t1}, there exists some $2\leq i \leq t$ such that $d_i \leq i$ and $d_{k-i} < k-i$. Let $r = r(G_k)$
 be the smallest such $i$. 

Because $G_k$ has  $r$ vertices of degree at most $r$, similarly to~\cite{Erd62},
 \[
	e(G_k) \leq r^2 + {k-r \choose 2}.
\] 
For $k=2t+1$, $r^2 + {k-r \choose 2} > h(n,k,t-1)$ only when $ r = t$ or $r < (t+4)/3$, and for $k=2t+2$, 
when $r = t$ or $r<(t+6)/3$. If $r = t$, then repeating the argument  in~\cite{main} yields:

\begin{lem}[\cite{main}]\label{rt} If $r(G_k) = t$ then $G_k \subseteq H_{k,k,t}$. 
\end{lem}

Then by Lemma~\ref{co28}, 
Lemma~\ref{oldmain2}, and Lemma~\ref{oldmain3}, $G\subseteq H_{n,k,t}$ and contains some subgraph in $\mathcal F(k)$. 
So we may assume that 
\begin{equation}\label{rt3}
\mbox{if $k=2t+1$ then $r < (t+4)/3$, and if $k=2t+2$ then $r < (t+6)/3$.} 
\end{equation}

Our next goal is to show that $G$ contains a large ``core'', i.e.,  a subgraph with large minimum degree. 
For this, we recall the notion of \emph{disintegration} used by Kopylov$~\cite{Kopy}$.

{\bf Definition}: For a natural number $\alpha$ and a graph $G$, the \emph{$\alpha$-disintegration} of 
a graph $G$ is the  process of iteratively removing from $G$ the vertices with degree 
at most $\alpha$  until the resulting graph has minimum degree at least $\alpha + 1$. 
This resulting subgraph $H = H(G, \alpha)$ will be called the $\alpha$-{\em core} of $G$. It is well known
that $H(G, \alpha)$ is unique and does not depend on the order of vertex deletion. 

\begin{claim}\label{dis} The $t$-core $H(G,t)$ of $G$ is not empty.
\end{claim}

{\bf Proof of Claim~\ref{dis}}: 
We may assume that for all $m\leq j<n$, graph $G_j$ was obtained from $G_{j+1}$ by contracting edge $x_jy_j$, where
$d_{G_{j+1}}(x_j)\leq d_{G_{j+1}}(y_j)$. By Rule (MBP2), 
$d_{G_{j+1}}(x_j)=\delta({G_{j+1}})$, provided that $\delta({G_{j+1}})\leq t-1$.

By definition, $|V_{\leq r}(G_k)|\geq r$. So by Lemma~\ref{hhh} (applied several times), for each $k+1\leq j\leq k+t-r$, because each $G_j$ is not a complete graph (otherwise it would have a hamiltonian cycle),
\begin{equation}\label{m281}
\delta(G_j)\leq j-k+r-1\;\mbox{ and} \; |V_{\leq j-k+r}(G_j)|\geq j-k+r.
\end{equation}
To show that 
\begin{equation}\label{m282}
\delta(G_j)\leq t-1\;\mbox{for all}\; k\leq j\leq n,
\end{equation}
 by~\eqref{m281} and~\eqref{rt3}, it is enough to observe that
$$\delta(G_j)\leq j-k+r-1\leq (n-k)+r-1\leq \frac{t-1}{2}+\frac{t+6}{3} -1 =\frac{5t+3}{6}< t.
$$

We will apply a version of $t$-disintegration in which we first  manually remove a sequence of vertices and count the number of edges they cover.
By~\eqref{m282} and (MBP2), $d_{G_{n}}(x_{n-1})=\delta({G_{n}})\leq n-k+r-1$. Let
 $v_{n}:=x_{n-1}$. 
 Then  $G - v_{n}$  is a subgraph of $G_{n-1}$. If $x_{n-2}\neq x_{n-1}*y_{n-1}$ in $G_{n-1}$, then
 let $v_{n-1}:=x_{n-2}$, otherwise let $v_{n-1}:=y_{n-1}$. In both cases, $d_{G-v_n}(v_{n-1})\leq n-k+r-2$.
We continue in this way until $j=k$: each time we delete from 
  $G - v_{n}-\ldots-v_{j+1}$ the unique survived vertex $v_j$ that was in the preimage of $x_{j-1}$ when we obtained $G_{j-1}$ from $G_j$.
  Graph $G - v_{n} - ... - v_{k+1}$ has $r \geq 2$ vertices of degree at most $r$. 
We additionally delete 2 such vertices $v_{k}$ and $v_{k-1}$. Altogether, we have lost at most $(r+n-k-1) + (r+n-k-2) + ...+ r + 2r $ 
edges in the deletions. 

Finally, apply $t$-disintegration to the remaining graph on $k-2 \in \{2t-1,2t\}$ vertices. Suppose that the resulting graph is empty.

{\bf Case 1:} $n = k$. Then \[e(G) \leq r+r + t(2t-1 -t) + {t \choose 2},\]
where $r+r$ edges are from $v_k$ and $v_{k-1}$, and after deleting $v_k$ and $v_{k-1}$, every vertex deleted removes at most $t$ edges, until we reach the final $t$ vertices which altogether span at most ${t \choose 2}$ edges. 

For $k=2t+1$, 
\[h(k,k,t-1) - e(G) \geq {2t+1-(t-1) \choose 2} + (t-1)^2 - \left[r+r + t(2t-1 -t) + {t \choose 2}\right] = t+2 - 2r,\]
which is nonnegative for $r < (t+3)/3 $. Therefore $e(G) 
 \leq h(k,k,t-1)$, a contradiction.

Similarly, if $k= 2t+2$, \[e(G) \leq r+r + t(2t -t) + {t \choose 2}, \text{and}\] 
\[h(k,k,t-1)  - e(G) \geq {2t+2-(t-1) \choose 2} + (t-1)^2 - [r+r + t(2t -t) + {t \choose 2}] = t+4 - 2r,\] which is nonnegative when $r<(t+6)/3$.

{\bf Case 2:} $k < n \leq k+(t-1)/2$. Then for $k=2t+1$, 
\begin{align*}
e(G) & \leq \big[(r+n-k-1) + (r+n-k-2) + \ldots + r\big] + 2r+ t(2t -1 -t) + {t \choose 2} \\
& \leq \big[(t-1) + (t-1) + \ldots + (t-1)\big] + h(k,k,t-1)\\
& = (t-1)(n-k) + h(k,k,t-1)\\
&= h(n,k,t-1),
\end{align*}
where the last inequality holds because $r+n-k-1\leq t-1$.

Similarly, for $k = 2t+2$, 
\begin{align*}
e(G) & \leq \big[(r+n-k-1) + (r+n-k-2) + \ldots + r\big] + 2r+ t(2t -t) + {t \choose 2}\\
& \leq (n-k)(t-1) + h(k,k,t-1)\\
& = h(n,k,t-1).
\end{align*} 
This contradiction completes the proof of Claim~\ref{dis}.
\qed

For the rest of the proof of Theorem~\ref{main}$'$, we will follow the method of Kopylov in~\cite{Kopy} to
show that $G \subseteq H_{n,k,2}$. Let $G^*$ be the $k$-closure of $G$.
 That is, add edges to $G$ until adding any additional edges creates a cycle of length at least $k$. 
In particular, for any non-edge $xy$ of $G^*$, there is an $(x,y)$-path in $G^*$ with at least $k-1$ edges.

Because $G$ has a nonempty $t$-core, and $G^*$ contains $G$ as a subgraph, $G^*$  also has a nonempty $t$-core 
(which  contains the $t$-core of $G$). Let $H = H(G^*,t)$ denote the $t$-core of $G^*$. We will show that
\begin{equation}
\mbox{\em
$H$ is a complete graph.}
\end{equation}
Indeed, suppose there exists a nonedge in $H$. Choose a longest path $P$ of $G^*$ whose terminal vertices $x \in V(H)$ and $y \in V(H)$ are nonadjacent. 
By the maximality of $P$, every neighbor of $x$ in $H$ is in $P$, similar for $y$. Hence $d_P(x) + d_P(y) = d_H(x) + d_H(y) \geq 2(t+1) \geq k$, and also $|P| = k-1$ (edges). 
By Kopylov' Theorem~\ref{le:kop}, $G^*$ must have a cycle of length at least $k$, a contradiction. 

Therefore $H$ is a complete subgraph of $G^*$. Let $\ell = |V(H)|$. Because every vertex in $H$ has degree at least $t+1$, $\ell \geq t+2$. 
Furthermore, if $\ell \geq k-1$, then $G^*$ has a clique $K$ of size at least $k-1$. Because $G^*$ is $2$-connected, we can extend a $(k-1)$-cycle 
of $K$ to include at least one vertex in $G^*-H'$, giving us a cycle of length at least $k$. It follows that 
\begin{equation}\label{ell}
\mbox{$t+2 \leq \ell \leq k-2$,}
\end{equation}
and therefore $k-\ell \leq t$. Apply a weaker $(k-\ell)$-disintegration to $G^*$, and denote by $H'$ the resulting graph. 
By construction, $H \subseteq H'$. 

\emph{Case 1}: There exists $v\in V(H')-V(H)$. Since $v\notin V(H)$, there exists a nonedge between a vertex in $H$ and a vertex in $H' - H$. Pick a longest path $P$ with  terminal vertices $x \in V(H')$ 
and $y \in V(H)$. Then $d_P(x) + d_P(y) \geq (k-\ell + 1) + (\ell - 1) = k$, and therefore $c(G^*) \geq k$.  

\emph{Case 2}: $H = H'$. Then \[e(G^*) \leq {\ell \choose 2} + (n - \ell)(k - \ell) = h(n, k, k-\ell).\] If $3\leq (k - \ell) \leq t-1$, 
then $e(G) \leq \max\{h(n,k,3), h(n, k, t-1)\}$, so by~\eqref{ell}, $k-\ell =2$, and $H$ is the complete graph with $k-2$ vertices. 
Let $D = V(G^*) - V(H)$. If there is an edge $xy$ in $G^*[D]$, then because $G^*$ is $2$-connected, there exist two vertex-disjoint paths, $P_1$ and $P_2$, 
from $\{x,y\}$ to $H$ such that $P_1$ and $P_2$ only intersect $\{x,y\} \cup H$ at the beginning and end of the paths. 
Let $a$ and $b$ be the terminal vertices of $P_1$ and $P_2$ respectively that lie in $H$. Let $P$ be any $(a,b)$-hamiltonian path of $H$. Then $P_1 \cup P \cup P_2 + xy$ is a cycle of length at least $k$ in $G^*$, a contradiction.

Therefore $D$ is an independent set, and since $G^*$ is $2$-connected, each vertex of $D$ has degree 2. Suppose there exists $u,v \in D$ where $N(u) \neq N(v)$. Let $N(u) = \{a,b\}, N(v) = \{c,d\}$ where it is possible that $b = c$. Then we can find a cycle $C$ of $H$ that covers $V(H)$ which contains edges $ab$ and $cd$. Then $C - ab - cd + ua + ub + vc + vd$ is a cycle of length $k$ in $G^*$. Thus for every $v \in D$, $N(v) = \{a,b\}$ for some $a,b \in H$. I.e., $G^* = H_{n,k,2}$, and thus $G \subseteq H_{n,k,2}$. 
\qed

\subsection{ Proof of Theorem~\ref{main}$'$ for all $n$}
We use induction on $n$ with the base case $n \leq k + (t-1)/2$. Suppose $n \geq k + t/2$ and  for all $k \leq n' < n$, Theorem~\ref{main}$'$ holds. 
Let $G$ be a $2$-connected graph $G$ with  $n$ vertices such that 
\begin{equation}\label{m301}
\mbox{$e(G) > \max\{h(n,k,t-1), h(n,k,3)\}$ and $c(G) < k$.}
 \end{equation}
Apply one step of (BP). If (BP4) was applied (so neither (BP2) nor (BP3) applies to $G$), then $G_m = G$ (with $G_m$ defined as in the previous case). 
By Lemmas~\ref{oldmain4},~\ref{co28}, and~\ref{oldmain3}, the theorem holds.

Therefore we may assume that either (BP2) or (BP3) was applied. Let $G^-$ be the resulting graph. 
Then $c(G^-) < k$, and $G^-$ is $2$-connected.  

\begin{claim}\label{claim5.3}
\begin{equation}\label{m26}
e(G^-) > \max\{h(|V(G^-)|, k, t-1), h(|V(G^-)|, k, 3)\}.
\end{equation}
\end{claim} 
 
\begin{proof}
If (BP2) was applied, i.e., $G^- = G/uv$ for some edge $uv$, 
then $$e(G^-) \geq e(G) - (t-1) > h(n-1, k, t-1) \geq h(n-1, k, 3),$$ so~\eqref{m26} holds. Therefore we may assume that (BP3) was applied to obtain $G^-$. Then  $n \geq k+t-1$ and  $e(G)-e(G^-)={t+1 \choose 2} - 1$. So by~\eqref{m301},
\begin{equation}\label{m26e}
e(G^-) > h(n,k,t-1) - {t+1 \choose 2} + 1.
\end{equation}
The right hand side of~\eqref{m26e} equals 
$h(n-(t-1), k, t-1)+t^2/2 -5t/2 + 2$ which is at least $h(n-(t-1), k, t-1)$ for $t\geq 4$, proving the first part of~\eqref{m26}. 

We now show that also  $e(G^-) > h(n-(t-1),k,3)$. Indeed,   for $k = 2t+1$, 
$$
e(G^-) - h(n-(t-1), k, 3) >
 {t+2 \choose 2} + (t-1)(n - t-2) - {t + 1 \choose 2} + 1 $$
$$ - \left[{2t -2 \choose 2} + 3(n - (t-1) - (2t-2))\right] 
 \geq 0 \text{ when } n\geq 3t. $$
Similarly, for $k = 2t+2$, 
$$
e(G^-) - h(n-(t-1), k, 3) >
 {t+3 \choose 2} + (t-1)(n - t-3) - {t + 1 \choose 2} + 1 $$
$$ - \left[{2t -1 \choose 2} + 3(n - (t-1) - (2t-1))\right] 
 > 0 \text{ when } n\geq 3t+1.
$$

Thus if $n\geq 3t+1$, then~\eqref{m26} is proved.
But if $n \in \{3t-1,3t\}$ then by Remark 5.1, no graph to which (BP3)  applied may have more than $h(n,k,t-1)$ edges. \end{proof}

By~\eqref{m26}, we may apply induction to $G^-$. So 
 $G^-$ satisfies either (a) $G^- \subseteq H_{|V(G^-)|, n, 2}$, or 
(b) $G^-$ is contained in a graph in $\mathcal G(n,k) - H_{|V(G^-)|,k,2}$ and contains a subgraph $H \in \mathcal F(k)$. 
Suppose first that $G^-$ satisfies (b). If (BP3) was applied to obtain $G^-$ from $G$, then because $G^-$ contains a subgraph 
$H \in \mathcal F(k)$ and $G^- \subseteq G$, $G$ also contains $H$. If (BP2) was applied, then by Lemma~\ref{oldmain2},
 $G$ contains a subgraph $H' \in \mathcal F(k)$. In either case,  Lemma~\ref{oldmain3} implies that $G$ is a subgraph of a graph in
 $\mathcal G(n,k) - H_{n,k,2}$.
 
So we may assume that (a) holds, that is, $G^-$ is a subgraph of $H_{|V(G^-)|, n, 2}$.  Because $\delta(G^-) \leq 2$, $\delta(G) \leq 3$, 
and so $G$ has edges in at most $2 \leq t-2$ triangles. Therefore (BP2) was applied to obtain $G^-$, where $G/uv = G^-$. 
Let $D$ be an independent set of vertices of $G^-$ of size $(n-1) - (k-2)$ with $N(D) = \{a,b\}$ for some $a,b \in V(G^-)$. 
Since $T_{G^-}(xa), T_{G^-}(xb) \leq 1$ for every $x\in D$, we have that $T_G(uv) \leq 2$ with equality only if $T(G) = 2$ where $T(G) = \min_{xy \in E(G)}T_G(xy)$.
 
We want to show that $T_G(uv) \leq 1$. If not, suppose first that $u*v \in D \subseteq V(G^-)$. 
Then there exists $x \in D -u*v$, and $x$ and $u*v$ are not adjacent in $G^-$. Therefore $x$ was not in a triangle with $u$ and $v$ in $G$,
 and hence $T_G(xa) = T_{G^-}(xa) \leq 1$, so the edge $xa$ should have been contracted instead.  Otherwise if $u*v \notin D$, at least one of $\{a,b\}$, say $a$, is not $u*v$. If $T(G) = 2$, 
then for every $x \in D \subseteq V(G)$, $T_G(xa) = 2$, therefore each such edge $xa$ was in a triangle with $uv$ in $G$. Then 
$T_G(uv) \geq |D| = (n-1) - (k-2) \geq k + t/2 - 1 - k + 2 \geq 3$, a contradiction.   

Thus $T_G(uv) \leq 1$ and $e(G)\leq 2+e(G^-)\leq 2 + h(n-1,k,2) = h(n,k,2)$. But for $n\geq k+t/2$, we have $h(n,k,t-1)\geq h(n,k,2)$, a contradiction.
\qed

\medskip
{\bf Acknowledgment.} The authors thank Zolt\' an Kir\' aly for helpful comments.


\begin{thebibliography}{99}



%


\bibitem{Ch}
V. Chv\'atal,
On Hamilton's ideals.
J. Combinatorial Theory Ser. B {\bf 12} (1972), 163--168.

%

\bibitem{Erd62}
P. Erd\H{o}s,
Remarks on a paper of P\'osa,
Magyar Tud. Akad. Mat. Kutat\'o Int. K\"ozl. {\bf 7} (1962), 227--229.


\bibitem{ErdGal59}
P. Erd\H{o}s and T. Gallai,
On maximal paths and circuits of graphs,
Acta Math. Acad. Sci. Hungar. {\bf 10} (1959), 337--356.


\bibitem{FaudScheB}
R. J. Faudree and R. H. Schelp, Ramsey type results,
Infinite and Finite Sets, {Colloq. Math. J. Bolyai} {\bf 10}, (ed. A. Hajnal et al.), North-Holland,
Amsterdam, 1975, pp. 657--665.

\bibitem{FaudSche75}
R. J. Faudree and R. H. Schelp,
Path Ramsey numbers in multicolorings,
J. Combin. Theory Ser. B {\bf 19} (1975), 150--160.


\bibitem{main}
Z. F\"uredi, A. Kostochka, and J. Verstra\"ete.
Stability in the Erd\H{o}s--Gallai Theorem on cycles and paths,
  {\em JCTB} {\bf 121} (2016), 197--228.



\bibitem{FS224}
Z. F\"uredi and M. Simonovits,
  The history of degenerate (bipartite) extremal graph problems,
Bolyai Math. Studies {\bf 25} pp. 169--264,
Erd\H{o}s Centennial (L. Lov\'asz, I. Ruzsa, and V. T. S\'os, Eds.) Springer, 2013.
Also see: {\tt arXiv:1306.5167}.


\bibitem{Kopy}
G. N. Kopylov,
Maximal paths and cycles in a graph,
Dokl. Akad. Nauk SSSR {\bf 234} (1977),  19--21.
(English translation: Soviet Math. Dokl. {\bf 18} (1977), no. 3, 593--596.)


\bibitem{Lewin}
M. Lewin,
On maximal circuits in directed graphs.
J. Combinatorial Theory Ser. B {\bf 18} (1975), 175--179.


%



\bibitem{Woodall}
D. R. Woodall,
Maximal circuits of graphs I,
Acta Math. Acad. Sci. Hungar. {\bf 28} (1976), 77--80.






\end{thebibliography}
\end{document}